\documentclass[12pt,leqno,fleqn]{amsart}  
 \usepackage{amsmath,amstext,amssymb,amsxtra,paralist}
\usepackage{txfonts} 
\usepackage[T1]{fontenc}
\usepackage{lmodern}

 \usepackage{euler}   
          
\usepackage{amsrefs}

\usepackage{mathtools}
\mathtoolsset{showonlyrefs,showmanualtags}

\theoremstyle{plain} 
\newtheorem{lemma}[equation]{Lemma} 
\newtheorem{proposition}[equation]{Proposition} 
\newtheorem{theorem}[equation]{Theorem}

\theoremstyle{definition}
\newtheorem{definition}[equation]{Definition}

\theoremstyle{remark}
\newtheorem{remark}[equation]{Remark}

\newtheorem*{ack}{Acknowledgment}

\usepackage{hyperref} 
\hypersetup{
    colorlinks=true,       
    linkcolor=blue,          
    citecolor=magenta,        
    filecolor=magenta,      
    urlcolor=cyan           
}

\allowdisplaybreaks
\numberwithin{equation}{section}

\begin{document}
	\title[Two Weight Hilbert Transform]{Two Weight Inequality for the Hilbert Transform:\\ A Real Variable Characterization,  I}
\author[M.T. Lacey]{Michael T. Lacey}
\address{School of Mathematics \\
Georgia Institute of Technology \\
Atlanta GA 30332 }
\email{lacey@math.gatech.edu}
\thanks{M.L. supported in part by the NSF grant 0968499, and a grant from the Simons Foundation (\#229596 to Michael Lacey), 
and the Australian Research Council through grant ARC-DP120100399.}
\author[E.T. Sawyer]{Eric T. Sawyer}
\address{ Department of Mathematics \& Statistics, McMaster University, 1280
Main Street West, Hamilton, Ontario, Canada L8S 4K1 }
\email{sawyer@mcmaster.ca}
\thanks{E.T.S. Research supported in part by NSERC}
\author[C.-Y. Shen]{Chun-Yen Shen}
\address{Department of Mathematics\\ National Central University \\ Chungli, 32054,                                                                                                         
Taiwan }
\email{chunyshen@gmail.com}
\author[I. Uriarte-Tuero]{Ignacio Uriarte-Tuero}
\address{ Department of Mathematics \\
Michigan State University \\
East Lansing MI }
\email{ignacio@math.msu.edu}
\thanks{I.U.-T. has been partially supported by grants DMS-0901524, DMS-1056965 (US NSF),
MTM2010-16232, MTM2009-14694-C02-01 (Spain), and a Sloan Foundation Fellowship.}

\begin{abstract}
Let $\sigma $ and $w $ be locally finite positive Borel measures on $\mathbb{R}$ which do not share a common point mass.  
Assume that the pair of weights satisfy a  Poisson $ A_2$ condition, and satisfy the testing conditions below, for the Hilbert 
transform $ H$, 
\begin{equation*}
\int _{I}   H (\sigma  \mathbf 1_{I}) ^2 \;d w \lesssim \sigma (I) \,,  
\qquad 
\int _{I}   H (w  \mathbf 1_{I}) ^2\; d \sigma  \lesssim w (I) \,,  
\end{equation*}
with constants independent of the choice of interval $ I$.  Then $ H (\sigma \,\cdot )$ maps $ L ^2 (\sigma )$ to $ L ^2 (w)$, verifying a conjecture of Nazarov--Treil--Volberg. 
The proof has two components, a `global to local' reduction, carried out in this paper, and an analysis of the `local' problem, carried out in Part II of this series.
\end{abstract}


\maketitle
\setcounter{tocdepth}{1}
\tableofcontents
  
\section{Introduction}
Define a truncated Hilbert transform of  a locally bounded signed measure $ \nu $ by  
\begin{equation*}
H _{\epsilon , \delta } \nu  (x) := \int _{ \epsilon < \lvert  y-x\rvert < \delta  } 
\frac{d\nu ( y) }{y-x}, \qquad 0 < \epsilon < \delta . 
 \end{equation*}
Given weights (i.e.\thinspace 
locally bounded positive Borel measures) $\sigma $ and $w $ on the real
line $\mathbb{R}$, we consider the following \emph{two weight norm
inequality} for the Hilbert transform,%
\begin{equation}
	\sup _{0 < \epsilon < \delta } \int_{\mathbb{R}}\vert H _{\epsilon ,\delta } ( f\sigma ) \vert ^{2}dw
	\le \mathscr N ^2 \int_{\mathbb{R}}\vert f\vert ^{2}\; d\sigma ,
\qquad  f\in L^{2}( \sigma ) ,  \label{two weight Hilbert}
\end{equation} 
where $\mathscr{N}$ is the best constant in the inequality,   uniform over all truncations of the 
Hilbert transform kernel.     Below, we will write the inequality above as $ \lVert H (f \sigma )\rVert_{L^2 (w)} \le \mathscr N \lVert f\rVert_{L ^2 (w)}$, 
that is the uniformity over the truncation parameters is suppressed.

The primary question is to find a real variable characterization of this inequality, and the theorem below 
is an answer to the beautiful conjecture of Nazarov-Treil-Volberg, see  \cite{V}.  Set 
\begin{equation} \label{e:Poisson}
P (\sigma ,I) := \int _{\mathbb R } \frac {\lvert  I\rvert } { \lvert  I\rvert ^2 + \textup{dist} (x,I) ^2  } \; \sigma (dx), 
\end{equation}
which is approximately the Poisson extension of $ \sigma $ to the upper half plane, evaluated at $ (x_I, \lvert  I\rvert) $, 
where $ x_I$ is the center of $ I$. 

\begin{theorem}
\label{t:main}Let $\sigma $ and $w $ be locally finite
positive Borel measures on the real line $\mathbb{R}$ with no common point
masses. Then, the two weight inequality \eqref{two weight Hilbert} holds if and only if these 
three conditions hold uniformly over all intervals $ I$, 
\begin{gather} \label{e:A2}
	P( \sigma, I )P( w, I ) \leq \mathscr A_2, 
	\\  \label{e:testing}
\int_{I}\vert H( \mathbf{1}_{I}\sigma ) \vert
^{2}\; d w \leq\mathscr{T} ^2 \sigma ( I), \qquad 
\int_{I}\vert H( \mathbf{1}_{I}w ) \vert
^{2}\; d\sigma \leq \mathscr{T} ^2 w ( I)\,.
\end{gather} 
	There holds 
\begin{equation}  \label{e:Hdef}
	\mathscr{N}\approx \mathscr A_{2} ^{1/2}+\mathscr{T} =: \mathscr H\,,  
\end{equation}
where $ \mathscr A_2$ and $ \mathscr T$ are the best constants in the inequalities above. 
\end{theorem}

It is well known \cite{V} that the $ A_2$ condition is necessary for the norm inequality, and the inequalities \eqref{e:testing} are obviously necessary,  thus the content of the Theorem is the sufficiency of the $ A_2$ and testing inequalities.  
In this paper, we will carry out a `global to local' reduction in the proof of sufficiency, with the analysis of the `local' problem being carried out in 
part II of this series \cite{13014663}.  

The Nazarov-Treil-Volberg conjecture  has only been verified before under additional 
hypotheses on the pair of weights, hypotheses which are not necessary for the two weight inequality.  
The so-called pivotal condition of \cite{V} is not necessary, as was proved in \cite{10014043}. 
The pivotal condition is still an interesting condition: It is all that is needed to characterize the boundedness of the 
Hilbert transform, together with the maximal function in both directions. But, the boundedness of this triple of operators 
is decoupled in the two weight setting \cite{1109.2027}. 

Our argument has these attributes. 
Certain degeneracies of the pair of weights must be addressed, the   contribution of the innovative 2004  
	paper  of Nazarov-Treil-Volberg \cite{10031596}, also see \cite{V}, 
	which was further sharpened with the property of \emph{energy} in \cite{10014043}, a crucial property of the Hilbert transform.  This theme is further developed herein, 
	with notion of functional energy in \S\ref{s.core}.

 The proof should proceed through the analysis of the bilinear form $\langle H (\sigma f),  g w\rangle $, as  one expects certain paraproducts to appear.  Still, the paraproducts have no canonical form, suggesting that the proof be highly non-linear in $ f$ and $ g$.   
	The non-linear  point of view was initiated in \cite{11082319}, and is  central to 
	this paper.  A particular feature of our arguments is a repeated appeal to certain \emph{quasi-orthogonality} arguments, 
	providing (many)  simplifications over prior arguments.  For instance, we never find ourselves constructing auxiliary measures, 
	and verifying that they are Carleson, a frequent step in many related arguments. 
	
\bigskip

One can phrase a two weight inequality question for any operator $ T$, a question that became apparent with 
the foundational paper of Muckenhoupt \cite{MR0293384} on $ A_p$ weights for the maximal function. 
Indeed, the case of Hardy's inequality was quickly resolved by Muckenhoupt \cite{MR0311856}. The maximal 
function was resolved by one of us \cite{MR676801}, as well as  the fractional integrals, and, essential  for this paper,  Poisson integrals \cite{MR930072}.  
The latter paper established a result which closely paralleled the contemporaneous $ T1$ theorem of 
David and Journ\'e \cite{MR763911}.  This connection,  fundamental in nature,  was not fully appreciated until the innovative work 
of Nazarov-Treil-Volberg \cites{NTV4,NTV2,MR1998349} in developing a non-homogeneous theory of singular integrals.  
The two weight problem for dyadic singular integrals was only resolved recently \cite{MR2407233}. Partial information 
about the two weight problem for singular integrals \cite{ptv} was basic to the resolution of the $ A_2$ conjecture \cite{1007.4330}, 
and several related results \cites{1103.5229,MR2657437,ptv,ptv2}. Our result is the first real variable characterization 
of a two weight inequality for a continuous singular integral.

Interest in the two weight problem for the Hilbert transform arises from its natural occurrence in questions related to
operator theory \cites{MR1945291,MR1207406}, spectral theory \cite{MR1945291}, and model spaces \cite{MR2198367},
and analytic function spaces \cite{MR1617649}.  In the context of operator theory Sarason posed the conjecture (See \cite{MR1334345}.) that 
the Hilbert transform would be bounded if the pair of weights satisfied the  (full) Poisson $ A_2$ condition. This was disproved 
by Nazarov \cite{N1}.   Advances on these questions have been linked to finer understanding of the two weight question, 
see for instance \cites{NV,MR1945291}, which build upon Nazarov's counterexample.

\begin{ack}
The authors benefited from a stimulating conference on two weight inequalities at the American Institute of Mathematics, 
Palo Alto California, in October 2011.  The reviewing process has lead to many improvements in this paper, including a 
streamlining of the main contribution of this paper, for which we thank the referees.  
\end{ack}

\section{Dyadic Grids and Haar Functions} \label{s.dyadic}

\subsection{Choice of Truncation}
 
We have stated the main theorem with `hard' cut-offs in the truncation of the Hilbert transform. 
There are many possible variants in the choice of truncation, moreover the proof of sufficiency  requires a different choice of truncation.  

Consider a  truncation given by  
\begin{align*}
\tilde H _{\alpha , \beta } (\sigma f) (x) &:= 
\int f (y)   K _{\alpha ,\beta } (y-x) \;{\sigma (dy)} 
\end{align*}
where $  K _{\alpha ,\beta } (y)$ is  chosen to minimize the technicalities associated with off-diagonal considerations. 
Specifically, set $  K _{\alpha ,\beta } (0) = 0$, and otherwise $  K _{\alpha ,\beta } (y)$ is odd and for $ y>0$
\begin{equation}\label{e:Kdef}
K _{\alpha ,\beta } (y) := 
\begin{cases}
-\frac y {\alpha ^2 } + \frac2 \alpha  & 0 < y < \alpha , 
 \\
\qquad \frac 1 y  & \alpha \leq  y \leq  \beta ,
\\
 -\frac y {\beta ^2 } + \frac2 \beta  &   \beta < y < 2 \beta , 
\\
\qquad 0 & 2 \beta \leq y . 
\end{cases}
\end{equation}
This is a $ C ^{1}$ function on $ (0, 2 \beta )$, and is Lipschitz, convex and monotone on $ (0, \infty )$.  

We now argue that  we can use these truncations in the proof of the sufficiency bound of our main theorem.  

\begin{proposition}\label{p:A22} If the pair of weights $ \sigma , w$ satisfy the $ A_2$ bound \eqref{e:A2}, 
then, one has the uniform norm estimate with the `hard' truncations \eqref{two weight Hilbert} if and only if one has uniform norm  estimate for the 
`smooth' truncations, 
\begin{equation} 
\sup _{0< \alpha < \beta } \lVert  \tilde H _{\alpha , \beta } (\sigma f)\rVert_{w} \le \mathscr N \lVert f\rVert_{\sigma }. 
\end{equation}
 \end{proposition}

Indeed, $ \lvert  H _{\alpha , \beta } (\sigma f) - \tilde H _{\alpha , \beta } (\sigma f)\rvert \lesssim A _{\alpha } (\sigma \lvert  f\rvert ) 
+ A _{\beta } (\sigma \lvert  f\rvert )$, where these last two operators are `single-scale' averages, namely 
\begin{equation*}
A _{\alpha  } ( \sigma \phi ) (x) = \alpha ^{-1} \int _{ (x- 3 \alpha , x+ 3 \alpha )} \phi (y) \; \sigma (dy) . 
\end{equation*}
But, the (simple) $ A_2$ bound is all that is needed to provide a uniform bound on the operators $ A _{\alpha } (\sigma \phi )$.  So the proposition 
follows.

Henceforth we use the truncations $ \tilde H _{\alpha , \beta }$, and we suppress the tilde in the notation.  
The particular  choice of truncation is  motivated by this off-diagonal estimate on the kernels.  

\begin{proposition}\label{p:grad} Suppose that $ 2 \lvert  x-x'\rvert < \lvert  x-y\rvert  $, then 
\begin{gather}\label{e:grad}
\begin{split}
{K _{\alpha , \beta } (y-x') - K _{\alpha , \beta } (y-x) }  = C _{x,x',y} \frac {x'-x} { (y-x) (y-x')} , 
\\
\textup{where} \quad 
C _{x,x',y}  = 1 \qquad    2\alpha < \lvert  x-y\rvert < \tfrac 12  \beta , 
\end{split}
\end{gather}
and is otherwise positive and never more than $ 4$. 
\end{proposition}

\begin{proof}
The assumptions imply that $ y-x'$ and $y-x$ have the same sign. Assume, without loss of generality that  $0<   y-x' < y-x$.   
If $ 2\alpha < \lvert  x-y\rvert < \frac 12 \beta $, it follows that $ \alpha < \lvert  x'-y\rvert < \beta $,  
and so by the definition 
\begin{align*}
{K _{\alpha , \beta } (y-x') - K _{\alpha , \beta } (y-x) }  =  \frac 1 {y-x'} - \frac 1 {y-x} = \frac {x'-x} { (y-x) (y-x')} . 
\end{align*}
And, in the general case, there holds $ \lvert  \frac d {dt}K _{\alpha , \beta } (t)\rvert \le 4 t ^{-2}$, so that 
\begin{align*}
0 \leq  {K _{\alpha , \beta } (y-x') - K _{\alpha , \beta } (y-x) }
& \leq  \int ^{y-x} _{y-x'} \frac 4 {t ^2 } \; dt=4  \frac {x'-x} { (y-x) (y-x')} . 
\end{align*}
\end{proof}

\subsection{Dyadic Grids.}
A collection of intervals $ \mathcal G$ is a \emph{grid} if for all $ G,G'\in \mathcal G$, we have $ G\cap G' \in \{\emptyset , G, G'\}$.  
By a \emph{dyadic grid} we mean a grid  $ \mathcal D $ of intervals of $ \mathbb R $ such that 
for each interval $ I\in \mathcal D$, the subcollection 
$ \{ I' \in \mathcal D \;:\; \lvert  I'\rvert= \lvert  I\rvert  \}$ partitions $ \mathbb R $, aside from endpoints 
of the intervals.  In addition, the  left and right halves of $ I$, denoted by $ I _{\pm}$, are also in $ \mathcal D$.  

For $I\in \mathcal{D}$, the left and right halves  $I_{\pm}$ 
are referred to as the \emph{children} of $I$.  We denote by $\pi _{\mathcal{D}}I $ the unique interval
in $\mathcal{D}$ having $I$ as a child, and we refer to $\pi _{\mathcal{D}%
}I$ as the $\mathcal{D}$-parent of $I$.

We will work with subsets $ \mathcal F \subset \mathcal D$.  We say that $ I$ has \emph{$ \mathcal F$-parent} $ \pi _{\mathcal F}I=F$ 
if $ F\in \mathcal F$ is the minimal element of $ \mathcal F$ that contains $ I$.

\subsection{Haar Functions.}
Let $ \sigma $ be a weight on $ \mathbb R $, one that does not assign positive mass to any endpoint 
of a dyadic grid $ \mathcal D$. If $ I\in \mathcal D$ is such that $ \sigma $ assigns non-zero weight to 
both children of $ I$, the associated Haar function is 
\begin{align}\label{e.hs1}
	h_{I}^{\sigma} & := \sqrt{\frac{\sigma(I_{-}) \sigma( I_{+})} {\sigma( I)}} 
\biggl( - \frac{{I_{-}}}{ \sigma( I_{-})} + \frac{{I_{+}}}{\sigma( I_{+})} \biggr)\,.    
\end{align}
In this definition, we are identifying an interval with its indicator function, and we will do so  throughout the remainder of the paper.  
This is  an $ L ^2 (\sigma )$-normalized  function, and   has $ \sigma $-integral zero.
For any dyadic interval $ I_{0}$, it holds that 
$  \{\sigma (I_0) ^{-1/2} {I_0}\} \cup \{ h ^{\sigma} _I \;:\; I\in \mathcal D\,, I \subset I_0\}$ is an orthogonal basis for $ L ^2 ( I_0, \sigma )$.  

We will use the notations  $ \hat f (I) = \langle f, h ^{\sigma } _{I} \rangle _{\sigma } $, as well as  
\begin{align}\label{e.mart}
\Delta ^{\sigma} _{I}f &= \langle f, h ^{\sigma} _{I} \rangle _{\sigma} h ^{\sigma} _{I}
= {I _{+}} \mathbb E ^{\sigma} _{I _{+}} f + 
{I _{-}} \mathbb E ^{\sigma} _{I _{-}}f - {I} 
\mathbb E ^{\sigma} _{I} f \,.  
\end{align}
The second equality is the familiar martingale difference equality, and so 
we will refer to $ \Delta ^{\sigma} _{I} f$ as a martingale difference.  It implies the familiar telescoping identity 
$
	\mathbb E _{J} ^{\sigma }f = \sum_{ I \;:\; I\supsetneq J} \mathbb E _{J} ^{\sigma } \Delta ^{\sigma } _{I} f \,. 
$

For any function the \emph{Haar support} of $ f$ is the collection $ \{I  \in \mathcal D\;:\; \hat f (I) \neq 0  \}$.

\subsection{Good-Bad Decomposition} 

With a choice of dyadic grid $ \mathcal D$ understood, we 
 say that $ J\in \mathcal D$ is \emph{$ (\epsilon ,r)$-good} if and only if  for all intervals $ I \in \mathcal D$ with $ \lvert  I\rvert\ge 2 ^{r-1}  \lvert  J\rvert $,  the distance from $ J$ to the boundary of \emph{either child} of $ I$ is at least $ \lvert  J\rvert ^{\epsilon } \lvert  I\rvert ^{1- \epsilon }$.    

For $ f\in L ^2 (\sigma )$ we set $ P _{\textup{good}} ^{\sigma } f = \sum_{ \substack{I \in \mathcal D \\  \textup{$ I$ is $(\epsilon,r)$-good}} }  \Delta ^{\sigma } _{I  }f $.  The projection $ P _{\textup{good}} ^{w } g  $ is defined similarly. 
To make the two reductions below, one must  make a \emph{random} selection of grids, as is detailed in 
\cites{10014043,V}.  
 The use of random dyadic grids has been a basic tool since the foundational work of 
\cites{NTV4,NTV2,MR1998349}.   
Important elements of the suppressed construction of random grids  are  that 
\begin{enumerate}
\item It suffices to consider a single dyadic grid $ \mathcal D$.

\item For any fixed $ 0<\epsilon< \frac 12 $, we can choose integer $ r$ sufficiently large so that  it suffices 
	to consider $ f$ such that $ f= P _{\textup{good}} ^{\sigma } f$, and likewise for $ g \in L ^2 (w)$. 
	Namely,   it suffices to estimate the constant 
	below, for arbitrary dyadic grid $ \mathcal D$, 
	\begin{equation*}
  \lvert \langle H _{\sigma } f, g \rangle _{w} 
 	   \rvert \le \mathscr N  _{\textup{good}}  \lVert f\rVert_ \sigma  \lVert g\rVert_ w\,, 
\end{equation*}
where it is required that $ f = P ^{\sigma } _{\textup{good}} \in L ^2 (\sigma )$ and $ g = P ^{w} _{\textup{good}}\in L ^2 (w)$. 
\end{enumerate}
 
 That the functions are good is, 
at some moments, an essential property. We suppress it in notation, however 
taking care to emphasize in the text those places in which we appeal to the property of being good.  

A reduction, using  randomized dyadic grids, allows one the extraordinarily useful reduction in the next Lemma.   This is a well-known reduction, due to Nazarov--Treil--Volberg,  explained in full detail in the current setting, in \cite{10031596}*{\S4}.   Below, $ \mathscr H$ is as in \eqref{e:Hdef}, the 
normalized sum of the $ A_2 $ and testing constants. 

\begin{lemma}\label{l:good} For all sufficiently small $ \epsilon $, and sufficiently large $ r$, this holds.  
Suppose that for any dyadic grid $ \mathcal D$, such that 
 no endpoint of an interval $ I\in \mathcal D$ is a point mass for $ \sigma $ or $ w$,\footnote{
This set of dyadic grids that fail this condition have probability zero in  standard constructions of the random dyadic grids.}
 there holds 
\begin{equation} \label{e:Hfg}
\lvert  \langle H _{\sigma } P ^{\sigma } _{\textup{good}}f, P ^{w} _{\textup{good}}g \rangle _{w}\rvert \lesssim \mathscr H \lVert f\rVert_{\sigma } \lVert g\rVert_{w} \,.  
\end{equation}
Then, the same inequality holds without the projections $ P ^{\sigma } _{\textup{good}}$, and $ P ^{w } _{\textup{good}}$.
\end{lemma}

Inequality \eqref{e:Hfg} should be understood as an inequality, uniform over the class of smooth truncations of the Hilbert transform. 
But, we can suppress this in the notation without causing confusion. 
The bilinear form only needs to be controlled for \emph{$ (\epsilon ,r)$-good functions $ f$ and $ g$, goodness being defined with respect to a fixed dyadic grid.} 
Suppressing the notation, we write `good' for `$(\epsilon,r)$-good,' and  it is always   assumed that  the dyadic grid $ \mathcal D$ is fixed, and only good intervals are in the Haar support of $ f$ and $ g$, though  is also suppressed in the notation.

\section{The Global to Local Reduction} \label{s:global}

The goal of this section is to reduce the analysis of the bilinear form in \eqref{e:Hfg} to the local estimate, \eqref{e:BF}.
It is sufficient to assume that $ f$ and $ g$ are supported  on an interval $ I^0$; by trivial use of the interval testing condition, we can further assume that $ f$ and $ g$ are of integral zero in their respective spaces.  
Thus, $ f $ is in the linear span of (good) Haar functions $ h ^{\sigma } _I$ for $ I\subset I^0$, and similarly for $ g$, and 
\begin{equation*}
\langle H _{\sigma } f,g \rangle_w = \sum_{I, J \::\: I,J\subset I^0} 
\langle H _{\sigma } \Delta ^{\sigma }_I f, \Delta ^{w} _{J} g \rangle _w \,. 
\end{equation*}
The argument is independent of the choice of truncation that implicitly appears in the inner product above.

The double sum is broken into different summands.  Many of the resulting cases are elementary, and we summarize these 
estimates as follows.  Define the bilinear form 
\begin{equation*}
B ^{\textup{above}} (f,g) := \sum_{I \::\: I\subset I^0} \sum_{J \::\: J\Subset I_J} 
\mathbb E ^{\sigma } _{I_J} \Delta ^{\sigma }_I f \cdot   \langle H _{\sigma } I_J, \Delta ^{w} _{J} g \rangle _w 
\end{equation*}
where here and throughout, $ J\Subset I$ means $ J\subset I$ and $ 2 ^{r}\lvert  J\rvert \le \lvert  I\rvert $. 
In addition, the argument of the Hilbert transform, $ I_J$,  is the child of $ I$ that contains $ J$, so that $  \Delta ^{\sigma }_I f$ is constant on $ I_J$.
Define $ B ^{\textup{below}} (f,g) $ in the dual fashion.  

\begin{lemma}\label{l:above}  There holds, with the notation of \eqref{e:Hdef}, 
\begin{equation*}
\bigl\lvert  \langle H _{\sigma } f ,g \rangle _{w} - B ^{\textup{above}} (f,g)  - B ^{\textup{below}} (f,g) \bigr\rvert 
\lesssim \mathscr H \lVert f\rVert_{\sigma } \lVert g\rVert_{w}  \,. 
\end{equation*}
\end{lemma}

This is a common reduction in a proof of a $ T1$ theorem, and in the current context, it only requires  goodness of intervals and the $ A_2$ condition.  
For a proof, one can consult \cites{V,10031596}.  The Lemma is specifically phrased and proved in  this way in \cite{11082319}*{\S8}.

These definitions are needed to phrase the global to local reduction.   The following definition depends upon the essential 
energy inequality \eqref{ENG} in the next section.

\begin{definition}\label{d:energy} Given any  interval $ F_0$, define $ \mathcal F_{\textup{energy}} (F_0)$ to be  
	the maximal subintervals $ F \subsetneq F_0$ such that     
\begin{equation}\label{e.Estop}
 P(\sigma F_0, F) ^2 \mathsf E (w,F) ^2 w (F) > 10 C_0{\mathscr H} ^2 \sigma (F) \,, 
\end{equation}
where $ E (w,F)$ is defined in \eqref{e:Edef}, and $ C_0$ is the constant in Proposition~\ref{p:energy}.  
There holds $ \sigma (\cup \{F \::\: F\in \mathcal F (F_0)\}) \le \tfrac 1 {10} \sigma (F_0)$.  
\end{definition}

\begin{definition}\label{d:BF} Let $ I_0$ be an interval, and let $ \mathcal S$ be a collection of disjoint intervals contained in $ I_0$. 
A function $ f \in L ^2 _0 (I_0, \sigma )$ is said to be  \emph{uniform (w.r.t.\thinspace  $ \mathcal S$)} if 
these conditions are met: 
\begin{enumerate}
\item  Each energy stopping interval $ F\in \mathcal F _{\textup{energy}} (I_0)$ is contained in some $ S\in \mathcal S$. 
\item The function $ f$ is constant on each interval  $ S \in \mathcal S$. 
\item  For any interval $ I \subset I_0$ which is not contained in any $ S\in \mathcal S$, 
$  \mathbb E ^{\sigma }_I \lvert f \rvert \le 1 $.  
\end{enumerate}
We will say that $ g$ is \emph{weakly adapted} to a function $ f$ uniform w.r.t.\thinspace $ \mathcal S$, if 
$ J\Subset S$ for some  interval $ S\in \mathcal S$ implies that $ \langle g, h ^{w} _{J} \rangle_w = 0$.   We will also say that $ g$ is \emph{weakly adapted to $ \mathcal S$.} 
\end{definition}

  The constant $ \mathscr L$ is defined as the best constant in the \emph{local estimate}: 
\begin{equation}  \label{e:BF}
\lvert  B ^{\textup{above}} (f,g)\rvert \le \mathscr L \{\sigma (I_0) ^{1/2} +  \lVert f\rVert_{\sigma }\} \lVert g\rVert_{w} \,, 
\end{equation}
where $ f, g$ are of mean zero on their respective spaces, supported on an interval $ I_0$. 
Moreover,  $ f$ is  uniform  and $ g$ is weakly adapted to $ f$.  
The inequality above is homogeneous in $ g$, but not $ f$, since the term $ \sigma (I_0) ^{1/2} $ is motivated by the bounded averages property of $ f$.

\begin{theorem}\label{t:above}[Global to Local Reduction]   There holds 
\begin{equation*}
\lvert B ^{\textup{above}} (f,g)\rvert 
\lesssim \{\mathscr H + \mathscr L\} \lVert f\rVert_{\sigma } \lVert g\rVert_{w}  \,.
\end{equation*}
The same inequality holds for the dual form $ B ^{\textup{below}} (f,g) $.
\end{theorem}

A reduction of this type is a familiar aspect of many proofs of a $ T1$ theorem, proved by exploiting standard off-diagonal estimates for Calder\'on--Zygmund kernels, but in the current setting, it is a much deeper fact, a consequence of the \emph{functional energy inequality} of \S\ref{s.core}.	
We make the following construction for an  $ f \in L ^{2 } (I^0, \sigma )$,   of $ \sigma $-integral zero. 
Add $ I^0$ to $ \mathcal F$, and set $ \alpha _{f} (I^0) := \mathbb E ^{\sigma } _{I^0} \lvert  f\rvert $.  
In the inductive stage, if $ F\in \mathcal F$ is minimal, add to $ \mathcal F$ those maximal descendants $ F'$ of 
$ F$ such that $ F'\in \mathcal F _{\textup{energy}} (F)$ or $ \mathbb E ^{\sigma } _{F'} \lvert  f\rvert \ge 10 \alpha _{f} (F)  $. 
Then define 
\begin{equation*}
\alpha _{f} (F') := 
\begin{cases}
\alpha _{f} (F)   &  \mathbb E ^{\sigma } _{F'} \lvert  f\rvert < 2\alpha _{f} (F) 
\\
 \mathbb E ^{\sigma } _{F'} \lvert  f\rvert  & \textup{otherwise}
\end{cases}
\end{equation*}
If there are no such intervals $ F'$, the construction stops.  
We refer to $ \mathcal F$ and $ \alpha _{f} ( \cdot )$
 as  \emph{Calder\'on--Zygmund stopping data for $ f$}, following the terminology of 
\cite{11082319}*{Def 3.5}.  Their key properties  are collected here. 
 
\begin{lemma}\label{l:CZ}  For $ \mathcal F$ and $ \alpha _{f} ( \cdot )$ as defined above, there holds 
\begin{enumerate}

\item  $ I_0$ is the maximal element of $ \mathcal F$.

\item  For all $ I\in \mathcal D$, $ I\subset I^0$,  we have $   \mathbb E _{I} ^{\sigma } \lvert f \rvert  \le  10 \alpha _{f} (\pi _{\mathcal F} I) $. 

\item  $ \alpha _{f}$ is monotonic: If $ F, F'\in \mathcal F$ and $ F\subset F'$ then $ \alpha _{f} (F)\ge \alpha _{f} (F')$.  

\item  The collection $ \mathcal F$ is $ \sigma $-Carleson in that 
\begin{equation} \label{e:sCarleson}
	\sum_{F\in \mathcal{F}:\ F\subset S}\sigma (F) \leq   2\sigma (S), \qquad S\in \mathcal{D}.
\end{equation}

\item We have the inequality 
	\begin{equation} \label{e:Quasi}
\Bigl\lVert 
\sum_{F\in \mathcal F} \alpha _{f} (F) \cdot F
 \Bigr\rVert_{\sigma } \lesssim    \lVert f\rVert_{ \sigma }  \,. 
\end{equation}

\end{enumerate}

\end{lemma}

\begin{proof}
The first three properties are immediate from the construction.  The fourth, the $ \sigma $-Carleson property is seen this way. 
It suffices to check the property for $ S\in \mathcal F$. 
Now, the $ \mathcal F$-children can be in  $ \mathcal F _{\textup{energy}} (S)$, which satisfy 
\begin{equation*}
\sum_{F'\in  \mathcal F _{\textup{energy}} (S)} \sigma (F') \le \tfrac 1 {10} \sigma (S) \,. 
\end{equation*}
Otherwise,  note that by choice of $ \alpha _{f} ( \cdot )$, we have $ \mathbb E _{S} ^{\sigma } \lvert  f\rvert  \le 2\alpha _{f} (S) $. 
These intervals $ F'$,  satisfy $ \mathbb E ^{\sigma } _{F'} \lvert  f\rvert \ge 10 \alpha _{f} (S) \ge 5\mathbb E ^{\sigma } _{S} \lvert  f\rvert  $. 
These  intervals satisfy  the display above with $ \frac 1 {10} $ replaced by $ \tfrac 1 5$.   Hence, \eqref{e:sCarleson} holds.  

For the final property,  let $ \mathcal G \subset \mathcal F$ be the subset at which the stopping values change: 
If $ F \in \mathcal F - \mathcal G$, and $ G$ is the $ \mathcal G$-parent of $ F$, then $ \alpha _{f} (F)= \alpha _{f} (G)$.  
Set 
\begin{equation*}
\Phi _{G} := \sum_{F\in \mathcal F \::\: \pi _{\mathcal G} F=G} F \,. 
\end{equation*}
Define $ G_k := \{ \Phi _G \ge 2 ^{k}\}$, for $ k=0, 1 ,\dotsc $.  The $ \sigma $-Carleson property implies  
integrability of all orders in $ \sigma $-measure of $ \Phi _G$. 
Using the third moment, we have $ \sigma (G_k) \lesssim 2 ^{-3k} \sigma (G)$.  
Then, estimate 
\begin{align*}
\Bigl\lVert 
\sum_{F\in \mathcal F} \alpha _{f} (F) \cdot F
 \Bigr\rVert_{\sigma } ^2 
 &=
 \Bigl\lVert 
\sum_{G\in \mathcal G} \alpha _{f} (G) \Phi _G 
 \Bigr\rVert_{\sigma } ^2
 \\ &\le 
  \Bigl\lVert 
\sum_{k=0} ^{\infty }  (k+1) ^{+1-1}\sum_{G\in \mathcal G} \alpha _{f} (G) 2 ^{k} \mathbf 1_{G_k} 
 \Bigr\rVert_{\sigma } ^2
 \\
 & \stackrel \ast \lesssim 
\sum_{k=0} ^{\infty } 
(k+1) ^2  \Bigl\lVert 
\sum_{G\in \mathcal G} \alpha _{f} (G) 2 ^{k} \mathbf 1_{G_k} (x)
 \Bigr\rVert_{\sigma } ^2
\\
 &\stackrel {\ast \ast  } \lesssim 
\sum_{k=0} ^{\infty } 
(k+1) ^2   
\sum_{G\in \mathcal G} \alpha _{f} (G) ^2  2 ^{2k} \sigma (G_k)  
\\
& \lesssim \sum_{G\in \mathcal G} \alpha _{f} (G) ^2  \sigma (G)   \lesssim \lVert M f\rVert_{\sigma } ^2 \lesssim \lVert f\rVert_{\sigma } ^2 \,. 
\end{align*}
Note that we have used Cauchy--Schwarz in $ k$ at the step marked by an $ \ast $. 
In the step marked with $ \ast \ast $, for each point $ x$, the non-zero summands are a (super)-geometric sequence of scalars, so the square can be moved inside the sum. 
Finally, we use the estimate on the $ \sigma $-measure of $ G_k$, and compare to the maximal function $ M f$ to complete the estimate. 

\end{proof}

We will  use the notation 
\begin{equation*}
 P ^{\sigma } _{F} f := \sum_{ I \in \mathcal D \::\: \pi _{\mathcal F}I=F} \Delta ^{\sigma } _{I} f\,, \qquad F\in \mathcal F\,. 
\end{equation*} 
and similarly for  $ Q ^{w} _{F}$, but rather than use $ \pi _{\mathcal F} J$, in the definition, 
we use $ \dot \pi _{\mathcal F} J $, defined to be the minimal $ F\in \mathcal F$ with $ J\Subset F$. 
Without this alternate definition, some delicate case analysis would be forced upon us. 
The inequality \eqref{e:Quasi} allows us to estimate  
\begin{align}  \label{e:quasi} 
	\begin{split}
		\sum_{F\in \mathcal F} 
	\{ \alpha _{f} (F) \sigma (F) ^{1/2}& + \lVert P ^{\sigma } _{F} f\rVert_{\sigma }\} 
	\lVert Q ^{w} _{F} g\rVert_{w} 
\\	& \le 
	\Biggl[
	\sum_{F\in \mathcal F} 
	\{\alpha _{f} (F) ^2  \sigma (F) + \lVert P ^{\sigma } _{F} f\rVert_{\sigma } ^2 \} 
	\times 
	\sum_{F\in \mathcal F} 
	\lVert Q ^{w} _{F} g\rVert_{w} ^2 
	\Biggr] ^{1/2}   \lesssim   \lVert f\rVert_{\sigma } \lVert g\rVert_{w} \,. 
	\end{split}
\end{align}
We will refer to this as the \emph{quasi-orthogonality}  argument, and we remark that it only requires orthogonality 
of the projections $ Q ^{w} _{F} g$. 
It is very useful.

\begin{lemma}\label{l:aboveCorona}
There holds 
\begin{gather*}
\bigl\lvert  B ^{\textup{above}} (f,g) - 	B ^{\textup{above}} _{\mathcal F} (f,g) \bigr\rvert \lesssim \mathscr H 
\lVert f\rVert_{\sigma } \lVert g\rVert_{w} \,, 
\\
\textup{where } \quad 
	B ^{\textup{above}} _{\mathcal F} (f,g) := \sum_{F\in \mathcal F}  B ^{\textup{above}} (P ^{\sigma } _{F}f, Q ^{w} _{F}g)\,. 
\end{gather*}
\end{lemma}

\begin{proof}  We apply functional energy, of \S\ref{s.core}.  Observe that $ f = \sum_{F\in \mathcal F} P ^{\sigma } _{F} f$, and 
\begin{equation*}
\sum_{J \::\: J\Subset I_0} \Delta ^{w} _{J} g = \sum_{F\in \mathcal F} Q ^{w} _{F} g.  
\end{equation*}
From the definition of $ B ^{\textup{above}} (f,g)$, we can assume that $ g$ equals the sum above. Therefore, 
	\begin{align*}
		 B ^{\textup{above}} (f,g)& = 
		 \sum_{F'\in \mathcal F} \sum_{F\in \mathcal F}   B ^{\textup{above}} (P ^{\sigma } _{F'}f, Q ^{w} _{F}g). 
		 \end{align*}
In the sum above, we can also add the restriction that $ F'\cap F\neq \emptyset $, for otherwise $ B ^{\textup{above}} (P ^{\sigma } _{F'}f, Q ^{w} _{F}g)=0$. 
For a pair of intervals $ J\Subset I_J$, note that this implies that $ J\Subset \pi _{\mathcal F} I$, 
that is $ \dot \pi _{\mathcal F} J\subset \pi _{\mathcal F}I$.  
	 Therefore, we can add the restriction $ F\subset F'$. 
	 The case of $ F'=F$ is 
	 the definition of $ 	B ^{\textup{above}} _{\mathcal F} (f,g) $, so that it suffices to estimate 
\begin{equation}  \label{e:BB}
\sum_{\substack{F,F'\in \mathcal F\\ F'\supsetneqq F }}  B ^{\textup{above}} (P ^{\sigma } _{F'}f, Q ^{w} _{F}g)\,. 
\end{equation}

Observe that the functions 
$
g_F := Q ^{w} _{F}g 
$ 
are $ \mathcal F$ adapted in the sense of  Definition~\ref{F adapted}, and by construction $ \mathcal F$ satisfies 
the Carleson measure condition \eqref{e:sCarleson}.  
We take these steps  to apply functional energy inequality. 
 The argument of the Hilbert transform is $ I_F$, the child of $ I$ that contains $ F$. 
	Write $ I_F= F + (I_F-F)$, and use linearity of $ H _{\sigma }$.  Note that by the standard martingale difference identity and 
	the construction of stopping data, 
\begin{equation*}
	\Bigl\lvert \sum_{I \;:\; I\supsetneq  F} \mathbb E ^{\sigma } _{I _{F}} \Delta ^{\sigma } _{I} f  \Bigr\rvert \lesssim \alpha _{f} (F)\,, 
	\qquad F\in \mathcal F \,. 
\end{equation*}
Hence, invoking interval testing, 
\begin{align*}
	\Bigl\lvert \sum_{F\in \mathcal F} 
		\sum_{I \;:\; I\supsetneq  F} \mathbb E ^{\sigma } _{I _{F}} \Delta ^{\sigma } _{I} f \cdot 
		\langle H _{\sigma } F, g_F \rangle_w \Bigr\rvert & \lesssim
		\sum_{F\in \mathcal F}  \alpha _{f} (F) 
		\bigl\lvert \langle H _{\sigma } F, g_F \rangle_w \bigr\rvert
		\\
		& \lesssim \mathscr  H 	\sum_{F\in \mathcal F}  \alpha _{f} (F) \sigma (F) ^{1/2} \lVert g_F\rVert_{w} \,.
\end{align*}
 Quasi-orthogonality bounds this last expression. 

For the second expression, when the argument of the Hilbert transform is $ I_F - F$, first note that 
\begin{equation*}
\Bigl\lvert 
\sum_{I \;:\; I\supsetneq  F} \mathbb E ^{\sigma } _{I _{F}} \Delta ^{\sigma } _{I} f \cdot  (I_F - F)
\Bigr\rvert \lesssim \Phi := \sum_{F'\in \mathcal F} \alpha _{f} (F') \cdot F' \,, \qquad F\in \mathcal F \,. 
\end{equation*}
Therefore, by the definition of $ \mathcal F$-adapted, the monotonicity property \eqref{e.mono1} applies, and yields  
\begin{equation*}
\Bigl\lvert 
\sum_{I \;:\; I\supsetneq  F} \mathbb E ^{\sigma } _{I _{F}} \Delta ^{\sigma } _{I} f \cdot \langle H _{\sigma } (I_F - F),  g_F\rangle _{w}
\Bigr\rvert \lesssim \sum_{J\in \mathcal J ^{\ast} (F)}P (\Phi \sigma , J) 
\Bigl\langle  \frac {x} {\lvert  J\rvert } , J  \overline  g_F\Bigr\rangle_{w} \,, \qquad F\in \mathcal F\,. 
\end{equation*}
Here,  $ \mathcal J ^{\ast} (F)$ are the maximal good intervals $ J\Subset F$, and 
$ \overline g  _F := \sum_{J\in \mathcal J (F) \::\: J\Subset F} \lvert  \hat g (J) \rvert \cdot h ^{w} _{J} $, so that every term has a positive inner product with $ x$.  
The sum over $ F\in \mathcal F$ of this last expression is controlled by functional energy, and the property 
that $ \lVert \Phi \rVert_{\sigma } \lesssim \lVert f\rVert_{\sigma }$.  This completes the bound for \eqref{e:BB}.  

\end{proof}

\begin{proof}[Proof of Theorem~\ref{t:above}] 
By Lemma~\ref{l:aboveCorona},  it remains to control $ B _{\mathcal F} ^{\textup{above}} (f,g)$. 
Keeping the quasi-orthogonality argument in mind, we see that   appropriate control on the individual summands is enough to control it.
For each $ F\in \mathcal F$, let $ \mathcal S _{F}$ be the $ \mathcal F$-children of $ F$.   Observe that the function 
\begin{equation} \label{e:adap}
(C \alpha _{f} (F)) ^{-1} P ^{\sigma }_F f 
\end{equation}
is uniform on  $ F$ w.r.t.\thinspace $ \mathcal S_F$, for appropriate absolute constant $ C$. 
Moreover, the function $ Q ^{w} _{F}g$ does not have any interval $ J$ in its Haar support strongly contained in an interval $S\in \mathcal S_F $. 
That is, it is weakly adapted to the function in \eqref{e:adap}.  
Therefore, by assumption, 
\begin{equation*}
\lvert   B ^{\textup{above}} (P ^{\sigma } _{F}f, Q ^{w} _{F}g)\rvert 
\le \mathscr L \{\alpha _{F} (F) \sigma (F) ^{1/2} + \lVert P ^{\sigma }_F f\rVert_{\sigma }\} \lVert Q ^{w} _F g\rVert_{w}  \,. 
\end{equation*}
The sum over $ F\in \mathcal F$ of the right hand side is bounded by the quasi-orthogonality argument of \eqref{e:quasi}. 

\end{proof}

\section{Energy, Monotonicity, and Poisson} \label{s.prelim}

Our Theorem is particular to the Hilbert transform, and so depends upon 
special properties of it.  They largely extend from the 
fact that the derivative of $-1/y $ is positive.  
The following Monotonicity Property for the Hilbert transform  was observed in \cite{11082319}*{Lemma 5.8}, 
and is basic to the analysis of the functional energy inequality.   

\begin{lemma}[Monotonicity Property]
\label{mono}   Let $ K \supsetneq I$ be two intervals, and assume that $ \sigma $ does not have point masses at the end point of $ I$. 
Then, for any function $ g\in L ^2 (I, w)$, with $ w$-integral zero,  and $ \beta > 2 \lvert  K\rvert $, 
\begin{equation}\label{e:P<H}
P (\sigma \cdot (K-I), I) \bigl\langle \frac {x} {\lvert  I\rvert }, \overline g\bigr\rangle_{w} 
\lesssim \liminf _{\alpha \downarrow 0} 
\langle H _{\alpha , \beta } (\sigma (K-I)),  \overline  g\rangle _{w} . 
\end{equation}
Here, $ \overline g = \sum_{J'} \lvert \widehat g (J') \rvert h ^{w} _{J'}$ is a Haar multiplier applied to $ g$. 
If $ J$ is a good interval, $ J\Subset I$, then,  
for   function $ g \in L ^2  (J, w)$, with $ w$-integral zero, and signed measures $ \nu $ and $ \mu $ supported on 
$ K- I$, with $ \lvert  \nu \rvert \le \mu  $, 
it holds that 
\begin{equation} \label{e.mono1} 
	\sup _{0 < \alpha < \beta }\lvert  \langle H _{\alpha , \beta } \nu    , g \rangle _{w }\rvert
	\lesssim  
	P( \mu, J )
	\bigl\langle\frac  x { \vert J \vert }, \overline  g \bigr\rangle _{w }.
\end{equation}  
\end{lemma}

The truncations enter into the formulation of the lemma, since they play a notable role here. 
We need this preparation.  

\begin{lemma}\label{l:intervals} Let $ I$ and $ J$ be two intervals which share an endpoint $ a$, at which neither $ \sigma $ nor $ w$ have 
a point mass. Then, 
\begin{equation}\label{e:intervals}
\sup _{0 < \alpha < \beta } \lvert  \langle  H _{\alpha , \beta  } \sigma I,  J\rangle _{w}\rvert 
\lesssim \mathscr A_2 ^{1/2} \sqrt {\sigma (I) w (J)}. 
\end{equation}
\end{lemma}

\begin{proof}

If $ \lvert  I\rvert \simeq \lvert  J\rvert  $, this inequality is the weak boundedness principle of \cite{10014043}*{\S2.2}.  
So, let us assume that $ 10 \lvert  I\rvert < \lvert  J\rvert $.  Then, it remains to bound 
\begin{align*}
 \lvert  \langle  H _{\alpha , \beta  } \sigma I,  (J \setminus 10I)\rangle _{w}\rvert
& \le \sum_{n=11 } ^{\infty } \frac {\sigma (I)  w (J \cap ( (n+1 ) I \setminus n I))} {  n \lvert  I\rvert }  
\\
& \le \frac {\sigma (I) } {\lvert  I\rvert ^{1/2}  } P (w, I) ^{1/2} w (J) ^{1/2} \lesssim \mathscr A_2 ^{1/2} \sqrt {\sigma (I) w (J)}. 
\end{align*}
This depends upon obvious kernel bounds, and an application of Cauchy--Schwarz to derive the Poisson term above. 
\end{proof}

\begin{proof}[Proof of Lemma~\ref{mono}] 
By linearity, it suffices to prove \eqref{e:P<H} in the case of $ g = h ^{w} _{I}$.  The point is to separate the supports of the functions involved. 
Since $ I$ does not have a point mass at the end point of $ I$, we have  $ \sigma (\lambda I \setminus I) \downarrow 0$ as $ \lambda \downarrow 1$. 
It follows that we can fix a $ \lambda >1$ sufficiently small so that $ P (\sigma (K-I), I) \simeq P (\sigma (K- \lambda I), I)$, and one more condition that 
we will come back to.  Then, for $ 0 < \alpha < \tfrac 12 (\lambda-1) \lvert  I\rvert  $,  we estimate as below, where $ x_I$ is the center of $ I$,  
\begin{align*}
\langle  H _{\alpha , \beta } (\sigma (K - \lambda I)), h ^{w} _{I} \rangle _{w} 
& = 
\int _{K- \lambda I} \int _{I} 
\{  K _{\alpha , \beta}(y-x) - K _{\alpha ,\beta}(y-x_I)\} h ^{w} _{I} (x) \; w (dx) \, \sigma (dy) 
\\
&= \int _{K- \lambda I} \int _{I} 
 \frac {x-x_I} { (y- x) (y- x_J)} h ^{w} _{I} (x) \; w (dx) \, \sigma (dy) 
\\
& \gtrsim   P (\sigma (K-I), I) \bigl\langle \frac {x-x_I} {\lvert  I\rvert },h ^{w} _{I}  \bigr\rangle_{w}. 
\end{align*}
We have subtracted the term, since $ h ^{w} _{I} $ has integral zero, then applied \eqref{e:grad} with $ C _{x,x_J,y}=1$, as follows from our choices of $ \alpha $ and $ \beta $. 
Then, note that $ (x-x_J) h ^{w} _{I} \geq 0$, so that we can pull out the Poisson term.  The last line follows by  our selection of $ \lambda $ sufficiently close  to $ 1$.  
Then, the last condition needed, is to select $ \lambda $ sufficiently close to one that, in view of \eqref{e:intervals}, 
\begin{equation*}
\sup _{\alpha , \beta }\lvert  \langle  H _{\alpha , \beta } (\lambda I \setminus I), h ^{w} _{I}  \rangle _{w} \rvert
\lesssim \mathscr A_2 ^{1/2} \sqrt { \sigma (\lambda I \setminus I) } 
< c  P (\sigma (K-I), I) \bigl\langle \frac {x-x_I} {\lvert  I\rvert },h ^{w} _{I}  \bigr\rangle_{w}. 
\end{equation*}
In the last  line, $ c>0$ is an absolute constant.  This completes the proof of \eqref{e:P<H}.  

\smallskip

Turn to \eqref{e.mono1}.  The estimate \eqref{e:grad} applies.  
\begin{align*}
\lvert    \langle H _{\alpha , \beta } \nu    , g \rangle _{w }\rvert\rvert 
& = 
\Bigl\lvert 
\int_{K - I}\int _{J} \{ K _{\alpha , \beta}(y-x) - K _{\alpha , \beta}(y-x_J)\} 
h ^{w} _{J} (x) \; w (dx) \nu (dy)  
\Bigr\rvert
\\
&= \Bigl\lvert 
\int _{K - I}\int _{J}  C _{x,x_J,y} \frac { (x-x_J) } { (y-x) (y-x_J)} 
h ^{w} _{J} (x) \; w (dx) \nu (dy)  
\Bigr\rvert
\end{align*}
But recall that $ 0 \le C _{x,x_J,y} \leq 4$, and equals one for $ \alpha $ sufficiently small. 
Moreover, $ y-x$ and $ y- x_J$ have the same sign, and $(x-x_J) h ^{w} _{J} (x) \geq 0$.  
So 
an upper bound is obtained by passing from $ \nu $ to $ \mu $.  
\begin{align*}
\lvert    \langle H _{\alpha , \beta } \nu    , g \rangle _{w }\rvert 
&\le 
\int_{K - I}\int _{J}  \frac { (x-x_J) } { (y-x) (y-x_J)} 
h ^{w} _{J} (x) \; w (dx)  \mu  (dy) 
\\
& \simeq P (\mu , J) \bigl\langle \frac x {\lvert  J\rvert },  h ^{w }_{J}\bigr\rangle_{w}.
\end{align*}
\end{proof}

The concept of \emph{energy} is fundamental to the subject.  For interval $ I$, define 
\begin{equation}  \label{e:Edef}
	E (w,I) ^2 := \mathbb E ^{w (dx)} _{I} \mathbb E ^{w (dx')} _{I} 
	\frac { (x-x') ^2 } {\lvert  I\rvert ^2  }  = \frac2 {w (I)} \sum_{J\subset I} \bigl\langle \frac x {\lvert  I\rvert }, h ^{w} _{J}  \bigr\rangle_{w} ^2 \,. 
\end{equation}
Now, consider the \emph{energy constant}, the smallest constant $ \mathscr E$ such that this condition holds, as presented or 
in its dual formulation.   
For all dyadic intervals   $ I_0$, all partitions $ \mathcal P$ of $ I_0$ into dyadic intervals, it holds that 
\begin{equation} \label{ENG}
	\sum_{I\in \mathcal P} P (\sigma I_0, I) ^2 	E (w,I) ^2 w (I) \le \mathscr E ^2 \sigma (I_0)\,. 
\end{equation}
This was shown in \cite{10014043}*{Proposition 2.11} 

\begin{proposition}\label{p:energy} For a finite constant $ C_0$, 
 $ \mathscr E ^2 \le C_0\{\mathscr A_2 ^{1/2}   + \mathscr T\} ^2  = C_0\mathscr H ^2 $.  
\end{proposition}

We will always estimate $ \mathscr E$ by $ \mathscr H$.   The proof is recalled here. 

\begin{proof}
It suffices to consider the case of finite partitions $ \mathcal P$ of $ I$. We first prove a version of the energy inequality with `holes' in the argument of the Poisson.  It follows from \eqref{e:P<H} that we can fix $ 0 < \alpha < \beta $ such that 
\begin{equation*}
P (\sigma (I_0-I), I) ^2 	E (w,I) ^2 w (I)  
\lesssim  \lVert  H _{\alpha , \beta } ( \sigma (I_0-I))\rVert_{ L ^{2} (I, \sigma )} ^2 , \qquad  I\in \mathcal P. 
\end{equation*}
Then, using linearity  and interval testing, we have 
\begin{gather*}
\sum_{I\in \mathcal P} 
\lVert  H _{\alpha , \beta } (\sigma \cdot  I_0)\rVert_{ L ^{2} (I, \sigma )} ^2
\lesssim \lVert  H _{\alpha , \beta } (\sigma \cdot  I_0)\rVert_{ L ^{2} (I, \sigma )} ^2 \lesssim \mathscr H ^2 \sigma (I_0), 
\\ \textup{and} \qquad 
\sum_{I\in \mathcal P} 
 \lVert  H _{\alpha , \beta } ( \sigma \cdot I)\rVert_{ L ^{2} (I, \sigma )} ^2 
 \lesssim \mathscr H ^2 \sum_{I\in \mathcal P}  \sigma (I) \lesssim \mathscr H ^2 \sigma (I_0). 
\end{gather*}
Then, by the $ A_2$ bound, we have $ P (\sigma \cdot I, I) ^2 	E (w,I) ^2 w (I)   \lesssim \sigma (I)$, which we can sum over the partition. 
This completes the proof. 
\end{proof}

One should keep in mind that the concept of energy is related to the tails of the Hilbert transform. The energy inequality, 
and its multi-scale extension to the functional energy inequality, show that the control of the tails is very subtle in this problem.

We also need the following elementary Poisson estimate from \cite{V}; used  occasionally in this argument, it is crucial to the proof of Lemma~\ref{l:above}.  

\begin{lemma}
	\label{Poisson inequality}Suppose that $J\Subset I\subset I_0$, and that $ J$ is good.  Then 
\begin{equation}
\lvert J\rvert ^{2\epsilon -1}P(\sigma  ({I_0 - I})  , J)\lesssim \lvert I\rvert ^{2\epsilon -1}P
(\sigma ({I_0- I}) , I).  \label{e.Jsimeq}
\end{equation}
\end{lemma}

\begin{proof}
	We have $ \textup{dist}(J, I_0 - I ) \ge \lvert  J\rvert ^{\epsilon } \lvert  I\rvert ^{1- \epsilon }  $, 
	so that for any $ x\in I_0  - I $, we have 
\begin{equation*}
	 \frac {\lvert  J\rvert ^{2 \epsilon }} {(\lvert  J\rvert + \textup{dist}(x, J) ) ^2 } \lesssim   
		\frac {\lvert  I\rvert   ^{2 \epsilon }} {(\lvert  I\rvert + \textup{dist}(x, I) ) ^2 } \,. 
\end{equation*}	
 Integrating this last expression, it follows that 
\begin{align*}
\lvert J\rvert ^{2\epsilon -1}P(\sigma \cdot ({I_0 - I}) , J) 
& = 
\lvert J\rvert ^{2\epsilon -1} \int _{I_0 - I} 
\frac {\lvert  J\rvert }{(\lvert  J\rvert + \textup{dist}(x, J) ) ^2 } \; d \sigma 
\\
& \lesssim \lvert  I\rvert ^{2 \epsilon } 
 \int _{I_0 - I} 
\frac { 1 }{(\lvert  J\rvert + \textup{dist}(x, J) ) ^2 } \; d \sigma \,. 
\end{align*}
And this proves the inequality. 

\end{proof}

\section{The Functional Energy Inequality} \label{s.core} 
We state an important  multi-scale  extension of the energy inequality \eqref{ENG}.

\begin{definition}
\label{F adapted}Let $\mathcal{F}$ be a collection of dyadic intervals.   
A collection of (good) functions 
$\{g_{F}\}_{F\in \mathcal{F}}$ in $L^{2}\left( w \right) $ is said to
be \emph{$\mathcal F$-adapted}     if 
for all $ F\in \mathcal F$,  the  Haar support of the function $ g_F$ is contained in 
$ \{J \::\: \dot \pi _{\mathcal F}J=F\}$.
\end{definition}

\begin{definition}
\label{functional energy}Let $\mathscr{F}$ be the smallest constant in the
inequality below, or its dual form.  The inequality holds for all non-negative $h\in L^{2}(\sigma )$, all $%
\sigma $-Carleson collections $\mathcal{F}$, and all $\mathcal F$-adapted
collections $\{g_{F}\}_{F\in \mathcal{F}}$: 
\begin{equation}
\sum_{F\in \mathcal{F}}\sum_{J^{\ast }\in \mathcal{J}^{\ast }(F)
}P(h\sigma, J^{\ast } )
\bigl\vert \bigl\langle \frac{x}{\lvert
J^{\ast }\rvert },g_{F}{{J^{\ast }}}\bigr\rangle _{w }\bigr\vert \leq \mathscr{F}\lVert h\rVert _{\sigma  }%
\biggl[ \sum_{F\in \mathcal{F}}\lVert g_{F}\rVert _{w}^{2}\biggr] ^{1/2}\,.  \label{e.funcEnergy}
\end{equation}%
Here $\mathcal{J}^{\ast }(F) $ consists of the \emph{maximal}
good intervals $J\Subset F$.  Note that the estimate is universal in $ h$ and $ \mathcal F$, separately. 
\end{definition}

This constant was  identified in \cite{11082319}, and is herein shown to be necessary from the $ A_2$ and interval 
testing inequalities.   Recall the definition of $ \mathscr H$ in \eqref{e:Hdef}. 

\begin{theorem}\label{t.functionalEnergy} Assume that $ \mathscr F$ satisfies \eqref{e:sCarleson},   
then,  $ \mathscr F \lesssim \mathscr H$.   
\end{theorem}

The first step in the proof is the domination of the constant $ \mathscr F$ by the best constant in a certain two weight inequality for 
the Poisson operator, with the weights being determined by $ w$ and $ \sigma $ in a particular way.  This is the decisive step, 
since there is a two weight inequality for the Poisson operator proved by one of us.  It reduces the full norm inequality 
to simpler testing conditions, which are in turn 
controlled by the $ A_2$ and Hilbert transform testing conditions.

\subsection{The Two Weight Poisson Inequality}  \label{s.necessary}
Consider the weight 
\begin{equation*}
	\mu \equiv \sum_{F\in \mathcal{F}}\sum_{J \in \mathcal{J} ^{\ast} (F) }\Bigl\Vert P_{F,J}^{w }\frac{x}{\lvert J\rvert }\Bigr\Vert _{w}^{2}\cdot
\delta _{(x_J,\lvert J\rvert )}\,. 
\end{equation*}
Here, $ P ^{w } _{F,J} := \sum_{J'   \;:\; J'\subset J,\ \dot \pi _{\mathcal F} J=F}  \Delta ^{w} _{J'} $. 
We can replace $x$ by $x-c$ for any choice of $c$ we wish; the projection is
unchanged. And $\delta _{q}$ denotes a Dirac unit mass at a point $q$ in
the upper half plane $\mathbb{R}_{+}^{2}$.
We prove the two-weight inequality for the Poisson integral: 
\begin{equation}
	\lVert \mathbb{P}( h \sigma )\rVert _{L^{2}(\mathbb{R}_{+}^{2},\mu )}\lesssim  \mathscr H
	\lVert   h\rVert _{ \sigma  }\,,  \label{two weight Poisson}
\end{equation}for all nonnegative $  h$. 
Above, $\mathbb{P}(\cdot )$
denotes the Poisson extension operator to the upper half-plane, so that in particular 
\begin{equation*}
	\Vert \mathbb{P}( h \sigma )\Vert _{L^{2}(\mathbb{R}_{+}^{2},\mu
	)}^{2}=\sum_{F\in \mathcal{F}}\sum_{J\in \mathcal{J} ^{\ast} (F)}\mathbb{P}\left( h \sigma \right) (x_J,\left\vert J\right\vert )^{2}\Bigl\Vert P_{F,J}^{w }\frac{x}{\left\vert J\right\vert }\Bigr\Vert _{w }^{2}\,,
\end{equation*}
where $ x_J$ is the center of the interval $ J$. 
The proof of Theorem~\ref{t.functionalEnergy} follows by duality.  

Phrasing things in this way brings a significant advantage: The characterization of the 
two-weight inequality for the Poisson operator, \cite{MR930072},  reduces the full norm 
inequality above to these testing inequalities. 
For any dyadic interval $ I \in \mathcal D$  
\begin{align}
	\int_{ \mathbb R ^2 _+}\mathbb{P}\left( \sigma\cdot I \right) ^{2}d\mu (x,t)  & \lesssim
	\mathscr H ^2  \sigma (I)\,,
\label{e.t1}
\\
\int _{\mathbb R }\mathbb{P} ^{ \ast}   (t{\widehat{I}}\mu )^{2}\sigma
(dx)& \lesssim \mathscr A_2\int_{\widehat{I}}t^{2} \; d\mu (x,t),  \label{e.t2}
\end{align}
where $\widehat{I}=I\times \lbrack 0,\lvert I\rvert ]$ is the box over $I$
in the upper half-plane, and $ \mathbb P  ^{ \ast }   $ is the dual Poisson operator 
\begin{equation*}
	\mathbb{P} ^{\ast } (t{\widehat{I}}\mu )=\int_{\widehat{I}}\frac{t^{2}}{t^{2}+\lvert x-y\rvert ^{2}}\mu (dy,dt)\,.
\end{equation*}
One should keep in mind that the intervals $ I$ are restricted to be  in our fixed dyadic grid, a reduction allowed as the 
integrations on the left in \eqref{e.t1} and \eqref{e.t2} 
are done over the entire space, either $ \mathbb R ^2 _+ $ or $ \mathbb R $. 
(Goodness of the intervals $ I$ above is not needed.) 
This reduction is critical to the analysis below. 

\begin{remark}
	A gap in the  proof of the Poisson inequality at \cite{MR930072}*{Page 542}
	can be fixed as in \cite{MR1175693} or  \cite{0807.0246}.
\end{remark}

\subsection{The Poisson Testing Inequality: The Core}

This subsection is concerned with a part of inequality  \eqref{e.t1}: 
Restrict the integral on the left to the set  $ \widehat I\subset \mathbb R ^2 _+ $. 
\begin{equation}\label{e.core}
	\int_{\widehat I }\mathbb{P}\left( \sigma\cdot I \right) ^{2}d\mu (x,t)   \lesssim
	\mathscr H ^2  \sigma (I)\,. 
\end{equation}  
Since $\left( x_J ,\left\vert J\right\vert
\right) \in \widehat{I}$ if and only if $J\subset I$,  we have
\begin{eqnarray*}
\int_{\widehat{I}}\mathbb{P}\left( \sigma\cdot I \right) (
x,t) ^{2}d\mu (x,t) 
&=&\sum_{F\in \mathcal{F}}\sum_{J\in \mathcal{J} ^{\ast} (F):\
J\subset I}\mathbb{P}\left( \sigma\cdot I \right) \left(
x_J ,\left\vert J\right\vert \right)
^{2}\Bigl\lVert P_{F,J}^{w }\frac{x}{\left\vert J\right\vert }\Bigr\rVert _{w}^{2} \\
\end{eqnarray*}

For each $ J$, 
\begin{align} \label{e.P<} 
	\begin{split}
\Bigl\lVert P_{F,J}^{w }\frac{x}{\left\vert J\right\vert }\Bigr\rVert _{w}^{2} 
&\leq \int_{J}\left\vert \frac{x- \mathbb E  _{J}^{w }x}{\left\vert J\right\vert }\right\vert ^{2}dw \left( x\right) = 2  
E\left( w,J \right) ^{2} w (J  ) \le 2 w (J)\,. 
	\end{split}
\end{align}

Let $ \mathcal F_0$ be the maximal $ F\in \mathcal F$ which are strictly contained in $ I$, 
and let $ \mathcal J ^{\sharp}$ be those dyadic $ J$ such that $ (x_J, \lvert  J\rvert )$ is in the 
support of $ \mu $, but has no parent in $ \mathcal F_0$.  
These intervals are necessarily disjoint.  
Observe that by \eqref{e.P<} and the energy inequality, 
\begin{align} \label{e:sharp}
\sum_{J\in \mathcal J ^{\sharp}} 
\mathbb P (\sigma  F) (x_J, \lvert  J\rvert ) ^2 \mu (x_J , \lvert  J\rvert )
& \lesssim 
\sum_{J\in \mathcal J ^{\sharp}}  
P (\sigma\cdot F, J ) ^2 E (w,J) ^2 w (J) \lesssim \mathscr H ^2 \sigma (F).  
\end{align}

We claim that 
\begin{equation}\label{e:F0} 
\sum_{F\in \mathcal F_0}  \int _{\hat F }
\mathbb P (\sigma (I \setminus F)) (x,t) ^2 \; d \mu (x,t)
\lesssim \mathscr H \sigma (I). 
\end{equation}
This is sufficient, since 
\begin{align*}
\int_{\widehat{I}}\mathbb{P}\left( \sigma\cdot I \right) (
x,t) ^{2}d\mu (x,t) 
& \lesssim 
\textup{LHS} \eqref{e:sharp} + \textup{LHS} \eqref{e:F0} 
+ 
\sum_{F\in \mathcal F_0}  
 \int _{\hat F }
\mathbb P (\sigma\cdot  F) (x,t) ^2 \; d \mu (x,t)
\\
& \lesssim \mathscr H ^2 \sigma (I) + 
+ 
\sum_{F\in \mathcal F_0}  
 \int _{\hat F }
\mathbb P (\sigma\cdot  F) (x,t) ^2 \; d \mu (x,t).  
\end{align*}
The individual terms in the last sum are set up for a 
recursive application of this inequality.  
Due to the Carleson condition \eqref{e:sCarleson}, this recursion will finish the proof. 

It remains to prove \eqref{e:F0}, which is another instance of the energy inequality. 
For an interval $ F_0\in \mathcal F_0$, and $ F\in \mathcal F$ strictly contained 
in $ F_0$, each interval $ J\in \mathcal J ^{\ast} (F)$ is contained in some $ J_0\in \mathcal J ^{\ast} (F_0)$. 
Then, the intervals $ F\in \mathcal F$ are not good, but $ J$ and $ J_0$ are good, hence  
\begin{align*}
\mathbb P (\sigma (I \setminus F_0)) (x_J, \lvert  J\rvert ) ^2 \mu (x_J, \lvert  J\rvert) 
&= 
\Bigl[\int _{I \setminus F_0} \frac { \lvert  J\rvert } { \lvert  J\rvert ^2 + \lvert  x-x_J\rvert  ^2 } \Bigr] ^2 
\Bigl\lVert  P ^{w} _{F,J} \frac x {\lvert  J\rvert } \Bigr\rVert_w ^2 
\\
& = 
\Bigl[\int _{I \setminus F_0} \frac { 1 } { \lvert  J\rvert ^2 + \lvert  x-x_J\rvert ^2 } \Bigr] ^2 
\lVert  P ^{w} _{F,J}  x  \rVert_w ^2 
\\
& \lesssim 
\Bigl[\int _{I \setminus F_0} \frac { \lvert  J_0\rvert } { \lvert  J_0\rvert ^2 + \lvert  x-x_{J_0}\rvert ^2 } \Bigr] ^2 
\Bigl\lVert  P ^{w} _{F,J} \frac x {\lvert  J_0\rvert } \Bigr\rVert_w ^2 .  
\end{align*}
This follows from goodness:  For $ x \in I \setminus F_0$,  
\begin{equation*}
\lvert  J\rvert ^2 + \lvert  x-x_J\rvert^2
\ge  \lvert  x-x_J\rvert ^2 \ge 
\lvert  x-x_{J_0}\rvert ^2 \ge \lvert  J_0\rvert ^{\epsilon } \lvert  F_0\rvert ^{1- \epsilon }.   
\end{equation*}
But then, we can add the projections $ P ^{w} _{F,J}  $, due to orthogonality,  and use \eqref{e.P<} again to see that  
\begin{align*}
\sum_{\substack{F\in \mathcal F \\ F\subset F_0}} 
\sum_{\substack{J \in \mathcal J ^{\ast} (F) \\ J\subset J_0} }
\mathbb P (\sigma (I \setminus F_0)) &(x_J, \lvert  J\rvert ) ^2 \mu (x_J, \lvert  J\rvert)  
\\& \lesssim 
\mathbb P (\sigma \cdot I ) (x_ {J_0}, \lvert  J_0\rvert ) ^2 
\sum_{\substack{F\in \mathcal F \\ F\subset F_0}} 
\sum_{\substack{J \in \mathcal J ^{\ast} (F) \\ J\subset J_0} }
\Bigl\lVert  P ^{w} _{F,J} \frac x {\lvert  J_0\rvert } \Bigr\rVert_w ^2 
\\
& \lesssim 
\mathbb P (\sigma \cdot I ) (x_ {J_0}, \lvert  J_0\rvert ) ^2  E (w, J_0) ^2 w (J_0).  
\end{align*}
The sum over $ F_0\in \mathcal F_0$, and $ J_0 \in \mathcal J ^{\ast} (F_0)$ is 
controlled by the energy inequality.  This completes the proof of \eqref{e:F0}.

 \subsection{The Poisson Testing Inequality: The Remainder}
Now we turn to proving the following estimate for the global part of the
first testing condition \eqref{e.t1}:\begin{equation*}
\int_{\mathbb{R}_{+}^{2}- \widehat{I}}\mathbb{P}\left( \sigma\cdot I \right) ^{2}d\mu \lesssim \mathscr{A}_{2}
\sigma (I)\,.
\end{equation*}
Decompose the integral on the left into four terms:   With $ F_J$ the unique $ F\in \mathcal F$ with $ J\in \mathcal J ^{\ast} (F)$, 
and using \eqref{e.P<}, 
\begin{eqnarray*}
	\int_{\mathbb{R}_{+}^{2}- \widehat{I}} 
	\mathbb{P}\left( \sigma\cdot I \right) ^{2}d\mu &=&\sum_{J:\ \left( x_J ,\left\vert J\right\vert \right) \in \mathbb{R}_{+}^{2}- \widehat{I}}\mathbb{P}\left( \sigma\cdot I \right)
\left( x_J ,\left\vert J\right\vert \right)
^{2} \Bigl\Vert P_{F _{ J},J}^{w }\frac{x}{\lvert J\rvert }\Bigr\Vert _{w }^{2} \\
&\le &\Biggl\{ \sum_{\substack{ J \;:\;  J\cap 3I=\emptyset  \\ \left\vert
J\right\vert \leq \left\vert I\right\vert }}+\sum_{ J \;:\; J\subset
3I- I}+\sum_{\substack{J \;:\;  J\cap I=\emptyset  \\ \left\vert
J\right\vert >\left\vert I\right\vert }}+\sum_{ J \;:\; J\supsetneqq
I}\Biggr\}
\mathbb{P}\left( \sigma\cdot I \right) \left(  x_J ,\left\vert J\right\vert \right) ^{2} w (J)\\
&=&A+B+C+D.
\end{eqnarray*}

Decompose term $A$ according to the length of $J$ and its
distance from $I$,   to obtain:
\begin{align*}
A &\lesssim\sum_{n=0}^{\infty }\sum_{k=1}^{\infty }
\sum_{\substack{ J \;:\;  J\subset 3^{k+1}I- 3^{k}I  \\ \left\vert J\right\vert =2^{-n}\left\vert I\right\vert }}
\left( \frac{2^{-n}\left\vert
I\right\vert }{\textup{dist}\left( J,I\right) ^{2}}\sigma ( I)\right) ^{2}w ( J) \\
&\lesssim \sum_{n=0}^{\infty }2^{-2n}\sum_{k=1}^{\infty }
\frac{\left\vert
I\right\vert ^{2}\sigma ( I)w( 3^{k+1}I- 3^{k}I)}{\left\vert 3^{k}I\right\vert
^{4}}\sigma ( I) \\
&\lesssim \sum_{n=0}^{\infty }2^{-2n}\sum_{k=1}^{\infty }3^{-2k}\left\{ 
\frac{\sigma ( 3^{k+1}I)w(
3^{k+1}I)}{\left\vert 3^{k}I\right\vert ^{2}}\right\}
\sigma ( I)\lesssim \mathscr A_2\sigma ( I).
\end{align*}

Decompose term $B$ according to the length of $J$ and
then use the Poisson inequality \eqref{e.Jsimeq},  available to use because of goodness of intervals $ J$.  
We then obtain
\begin{align*}
	B &\lesssim \sum_{n=0}^{\infty }\sum_{\substack{ J \;:\;  J\subset
3I- I  \\ \left\vert J\right\vert =2^{-n}\vert \vert }} 2^{-n(2-4\epsilon) }\frac{\sigma ( I) ^2 }{\left\vert I\right\vert ^2 }  w ( J) 
\\
& \lesssim  \sum_{n=0}^{\infty } 2^{-n(2-4\epsilon) }\frac{\sigma ( 3I)w( 3I)}
{\vert 3I\vert^ 2}\sigma ( I)
\lesssim  
\mathscr A_{2}\sigma ( I).
\end{align*}

For term $C$,   for $ n = 1, 2 ,\dotsc, $, set $ \mathcal J_n$ to be those  good dyadic intervals $ J$ with $ \lvert  J\rvert > \lvert  I\rvert  $, 
$ J\cap I = \emptyset $,  and 
\begin{equation*}
(n-1) \lvert  J\rvert \le \textup{dist} (I,J) < n \lvert  J\rvert.   
\end{equation*}
These intervals have bounded overlaps.
Indeed, suppose that $ J_1 \subsetneq \cdots \subsetneq J _{r} $ are all members for $ \mathcal J_1$.  Then, by goodness, 
\begin{align*} 
\textup{dist} (J_1, I) & \geq \textup{dist}(J_r, I) \geq (n-1) 2 ^{r} \lvert  J_1\rvert + 
\textup{dist} (J_1 , \partial J _{r}) 
\\&
\geq \{(n-1) 2 ^{r} + 2 ^{r (1- \epsilon )} \} \lvert  J_1\rvert . 
\end{align*}
which is a contradiction to membership in $ \mathcal J_n$. 
Restricting the sum to intervals in $ \mathcal J_n$, there holds 
\begin{align*}
\sum_{J\in \mathcal J_n} 
\mathbb{P}( \sigma\cdot I ) (  x_J ,\vert J\vert ) ^{2} w (J) 
& \lesssim 
\sigma (I) ^2 \sum_{J\in \mathcal J_n}  
\frac { w (J)} { n ^{4} \lvert  J\rvert ^2  } 
\\
& \lesssim 
\frac  {\sigma (I) ^2 } {\lvert  I\rvert } \sum_{J\in \mathcal J_n}  
\frac { w (J) \cdot \lvert  I\rvert } { n ^{4} \lvert  J\rvert ^2  } 
\\
& \lesssim \frac {\sigma (I)} {n ^2 } \cdot \frac  {\sigma (I)} {\lvert  I\rvert } P (w, I) \lesssim \mathscr A_2  \frac {\sigma (I)} {n ^2 } . 
 \end{align*}
And this is summable in $ n \in \mathbb N $.

In the last  term $D$, all the intervals $ J$ contain $ I$. Note that 
\begin{align*}
\sum_{J \::\: J\supsetneq I} 
\mathbb{P}\left( \sigma\cdot I \right) \left(  x_J ,\left\vert J\right\vert \right) ^{2} w (J) 
& \lesssim 
\sigma (I) ^2  \sum_{J \::\: J\supsetneq I}  
\frac { w (J)} { \lvert  J\rvert ^2  } 
\\
& \lesssim \sigma (I) \cdot  \frac {\sigma (I)} { \lvert  I\rvert  }  \sum_{J \::\: J\supsetneq I}  \frac { w (J) \cdot \lvert  I\rvert } { \lvert  J\rvert ^2  } 
\\
& \lesssim \sigma (I) \cdot  \frac  {\sigma (I)} {\lvert  I\rvert } P (w, I) \lesssim \mathscr A_2  \sigma (I).  
\end{align*}

\subsection{The Dual Poisson Testing Inequality}

We are considering \eqref{e.t2}. Note that there is a power of $ t$ on both sides, and that the expressions on the two sides of this inequality are  
\begin{gather*}
\int_{\widehat{I}}t^{2}\mu (dx,dt)=\sum_{F\in \mathcal{F}}
\sum_ {\substack{{J\in  \mathcal{J}^{\ast}(F) }\\   J \subset I }}\lVert P_{F,J}^{w
}x\rVert _{w}^{2}\,,
\\
\mathbb{P} ^{\ast} ( t{\widehat{I}}\mu ) \left(
x\right) =\sum_{F\in \mathcal{F}}
\sum_ {\substack{{J\in  \mathcal{J}^{\ast}(F) }\\   J \subset I }}
\frac{\lVert P_{F,J}^{w }x\rVert
_{w}^{2}}{\lvert J\rvert ^{2}+\lvert x-x_J\rvert
^{2}} \,. 
\end{gather*}
We are to dominate $ \lVert \mathbb{P} ^{\ast } ( t{\widehat{I}}\mu ) \rVert_{\sigma} ^2 $ by the   first expression above.  
The squared norm will be the sum over integers $s $ of $ T_s$ below, in which 
the relative lengths of $ J $ and $ J'$ are fixed by $ s$. Suppressing the requirement that $ J, J'\subset I$, 
\begin{align*}
	T_{s}& :=\sum_{F\in \mathcal{F}}\sum_{\substack{ J\in \mathcal{J}^{\ast}(F) }}
	\sum_ {\substack{{F'\in \mathcal{F}} }} 
	\sum_{\substack{ J'\in \mathcal{J}^{\ast}(F)  \\ \lvert J'\rvert =2^{-s}\lvert J\rvert }}\int \frac{\lVert P^w_{F,J}x\rVert _{ w }^{2}}{\lvert J\rvert ^{2}+\lvert x-x_J\rvert ^{2}} \cdot \frac{\lVert 
	 P^w_{F',J'}  x\rVert _{ w}^{2}}{\lvert J'\rvert ^{2}+\lvert x-x_{J'}\rvert
^{2}}d\sigma \\
& \leq M_{s}\sum_{F\in \mathcal{F}}\sum_{\substack{ J\in \mathcal{J}^{\ast}(F)}}\lVert P^w_{F,J}x\rVert _{w }^{2} \\
\textup{where}\quad M_{s}& \equiv \sup_{F\in \mathcal{F}}\sup_{\substack{ J\in \mathcal{J}^{\ast}(F)}}
	\sum_ {\substack{{F'\in \mathcal{F}}}}
\sum_{\substack{ J'\in \mathcal{J}^{\ast}(F)  \\ \lvert J'\rvert =2^{-s}\lvert
J\rvert }}\int\frac{1}{\lvert J\rvert ^{2}+\lvert x-x_J\rvert ^{2}} \cdot \frac{ w (J') \cdot \lvert  J'\rvert ^2  }{\lvert J'\rvert ^{2}+\lvert
x-x_{J'}\rvert ^{2}}\; d\sigma.
\end{align*}
The estimate \eqref{e.P<} has been used in the definition of $ M_s$.  
We claim the term $M_{s}$ is at most a constant times $\mathscr A_22^{-s}$, and it is here that the full Poisson $ A_2$ condition is used.

Fix $ J$, and let $ n \in \mathbb N $ be the integer chosen so that $(n-1)\lvert J\rvert \leq \textup{dist}(J,J')\leq n\lvert J\rvert $.
Estimate the integral in the definition of $ M_s$ by 
\begin{equation*}
	\frac{w (J')}{\lvert J'\rvert 
}\int \frac{\lvert J'\rvert ^{2}}{\lvert J\rvert ^{2}+\lvert
x-x_J\rvert ^{2}} \cdot \frac{\lvert J'\rvert }{\lvert J'\rvert
^{2}+\lvert x-x_{J'}\rvert ^{2}}d\sigma\lesssim \mathscr A_2 {2^{-2s}} \,. 
\end{equation*}
This estimate is adequate for $ n=0,1,2$.   Then estimate the sum over $J'$ as follows. 
\begin{equation*}
\sum_{F'\in \mathcal{F}}\sum_{\substack{ J'\in \mathcal{J}^{\ast}(F')\;:\;\lvert J'\rvert =2^{-s}\lvert J\rvert  \\ (n-1)\lvert J\rvert \leq \textup{dist}(J,J')\leq n\lvert J\rvert }}{2^{-2s}} \lesssim {2^{-s}}\,.
\end{equation*}
because the relative lengths of $J$ and $J'$ are fixed, and each $ J'$ is in at most one $ \mathcal J ^{\ast} (F)$. 

For the case of $ n\geq  3$, restrict $ J'$ to be to the right of $ J$, and let $ t_n = \frac {x_J+ x _{J'}} 2 $, so that 
$ \lvert  x_J-t_n\rvert,\ \lvert  x _{J'} - t_n\rvert \simeq n \lvert  J\rvert  $. 
First,  estimate the integral  in the definition of $ M_s$ on the interval $ [t_n, \infty )$. 
\begin{equation*}
\frac{w (J')}{\lvert J'\rvert 
}\int _{t_n} ^{\infty } \frac{\lvert J'\rvert ^{2}}{\lvert J\rvert ^{2}+\lvert
x-x_J\rvert ^{2}} \cdot \frac{\lvert J'\rvert }{\lvert J'\rvert
^{2}+\lvert x-x_{J'}\rvert ^{2}}d\sigma\lesssim \mathscr A_2 \frac {2^{-2s}} {n ^2 } 
\end{equation*}
Then estimate the sum over $J'$ as follows. 
\begin{equation*}
\sum_{F'\in \mathcal{F}}\sum_{\substack{ J'\in \mathcal{J}^{\epsilon}(F')\;:\;\lvert J'\rvert =2^{-s}\lvert J\rvert  \\ (n-1)\lvert J\rvert \leq \textup{dist}(J,J')\leq n\lvert J\rvert }}\frac {2^{-2s}} {n ^2 }\lesssim  \frac {2^{-s}} {n ^2 }\,.
\end{equation*}
This is clearly summable in $ n\ge 4$. 

Now, estimate on the integral on the interval $ (- \infty , t_n)$, 
\begin{align*}
\frac{w (J')}{\lvert J'\rvert  }
\int _{ - \infty } ^{t_n } & \frac{\lvert J'\rvert ^{2}}{\lvert J\rvert ^{2}+\lvert
x-x_J\rvert ^{2}} \cdot \frac{\lvert J'\rvert }{\lvert J'\rvert
^{2}+\lvert x-x_{J'}\rvert ^{2}}d\sigma
\\&=
 2 ^{-2s} \frac{w (J')}{\lvert J\rvert  }
\int _{ - \infty } ^{t_n } \frac{\lvert J\rvert}{\lvert J\rvert ^{2}+\lvert
x-x_J\rvert ^{2}} \cdot \frac{\lvert J\rvert ^2  }{\lvert J'\rvert
^{2}+\lvert x-x_{J'}\rvert ^{2}}d\sigma 
\\
& \lesssim 
2 ^{-2s} 
\frac{w (J')}{ n ^2 \lvert J\rvert 
} P (\sigma ,J)\,. 
\end{align*}
Drop the term with the geometric decay in $ s$, and sum over $ n$ and $ J'$ to see that 
\begin{align*}
\sum_{n=4} ^{\infty } 
\sum_{F'\in \mathcal{F}}\sum_{\substack{ J'\in \mathcal{J}^{\epsilon}(F')\;:\;\lvert J'\rvert =2^{-s}\lvert J\rvert  \\ (n-1)\lvert J\rvert \leq \textup{dist}(J,J')\leq n\lvert J\rvert }} 
\frac{w (J')}{ n ^2 \lvert J\rvert 
} P (\sigma ,I) 
\lesssim  P (w, J) P (\sigma ,J) \lesssim  \mathscr A_2 \,.  
\end{align*}
Here, we have appealed to the full Poisson $ A_2$ condition. This completes the control of the dual Poisson testing condition.

\end{document}